\documentclass[12pt,twoside]{amsart}
\usepackage{amsmath}
\usepackage{amssymb}
\usepackage{wasysym}
\usepackage{psfrag}
\usepackage{graphicx,epsf,amsmath}  
\usepackage{epsf,graphicx}
\usepackage{comment}
  
\newtheorem{thm}{Theorem}[section]
\newtheorem{theorem}{Theorem}[section]
\newtheorem{pro}[thm]{Proposition}
\newtheorem{remark}[thm]{Remark}

\newtheorem{lemma}[thm]{Lemma}
\newtheorem{cor}[thm]{Corollary}

\def\RR{{\mathbb R}}

\title[conformal spectrum]{Conformal Spectrum and Harmonic maps}
\author[N. Nadirashvili]{Nikolai Nadirashvili}
\author[Y. Sire]{Yannick SIRE}

\begin{document}
\maketitle
\begin{abstract}
This paper is devoted to the study of the conformal spectrum (and more precisely the first eigenvalue) 
of the Laplace-Beltrami operator on a smooth connected compact Riemannian surface without boundary, endowed with a conformal class. We give a rather constructive proof of a critical metric which is smooth outside of a finite number of conical singularities and maximizes the first eigenvalue in the conformal class of the background metric. We also prove that there exists a map associating a point on the surface to a subfamily of eigenvectors associated to the first maximized eigenvalue which is harmonic from the surface to an Euclidean sphere. 
\end{abstract}
\tableofcontents
\section{Introduction}

Let $(M,g)$ be a smooth connected compact Riemannian surface without boundary. In this paper, we construct a map from the manifold $M$ into the sphere by means of eigenvectors of the first eigenvalue of the Laplace-Beltrami on $(M, \tilde g)$ where $\tilde g$ is conformal to $g$ and maximizes the first eigenvalue of the Laplace-Beltrami operator.  
More precisely, denote by $A_g(M)$ the area of the surface $(M,g)$ and denote $\Delta_g$ the Laplace-Beltrami operator on $(M,g)$. The spectrum of $-\Delta_g$ consists of the sequence $\left \{ \lambda_k(g) \right \}_{k \geq 0}$ and satisfies 
$$\lambda_0(g)=0 <\lambda_1(g) \leq \lambda_2(g) \leq ... \leq \lambda_k(g) \leq ...$$

If we assume that the area $A_g(M)$ is normalized by one then by the fundamental result of Korevaar (see \cite{korevaar} and also \cite{YangYau}), it follows that every $\lambda_k(g)$ for a given $k \geq 0$ has a universal bound depending on the topological type of $M$, over all the metrics $g$ with normalized area. 

More precisely, denote 
$$\Lambda(M)=\sup_{g} \lambda_1(g) A_g(M)$$
where the supremum is taken over all smooth Riemannian metrics $g$ on the manifold $M$. It is a well-known result that $\Lambda(M) <\infty$ and it has been proved in \cite{YangYau} that for an orientable surface of genus $\gamma$, we have  (see also \cite{korevaar})
$$\Lambda(M) \leq 8\pi (\gamma+1). $$
This allows to define a topological spectrum on $M$ for $-\Delta_g$ by taking upper bounds of the eigenvalue $\lambda_1$. 

In the last years, several works have been devoted to explicit computations of the quantity $\Lambda(M)$. For a surface of genus zero, Hersch (see \cite{hersch}) proved that 
$$\Lambda(\mathbb S^2)=8\pi.  $$  
In the case of non-orientable surfaces, Li and Yau \cite{LiYau} proved the following equality 
$$\Lambda(\mathbb R P^2)=12\pi$$
and as far as the quantity $\Lambda(M)$ is concerned, one of the author (see \cite{N1}) proved that 
$$\Lambda(\mathbb T^2)=\frac{8\pi^2}{\sqrt{3}}.$$

A result of Yang and Yau \cite{YangYau} ensures that 
 
$$ \Lambda (M) \leq 8\pi [\frac{\gamma+3}{2}]  $$ 

for any surface of arbitrary genus $\gamma$ and $[.]$ in the right hand side stands for the integer part. As far as the Klein bottle is concerned, we refer the reader to \cite{kleinNadir} and \cite{elsoufi}. 

The above discussion gives rise to two related problems:
\begin{itemize}
\item obtain a precise upper bound for $\Lambda(M)$ depending on the genus of the surface. 
\item obtain a sharp bound for $\lambda_1$ in a given conformal class of the surface.
\end{itemize}    
 Obviously any progress on each of these two problems gives information on the other one. 
\begin{remark}
Note that, following \cite{SI}, an extremal metric for the first eigenvalue is a critical point $g_0$ of the functional $\lambda_1$, i.e. for any analytic deformation $g_t$ of the Riemannian metric $g_0$ in the class of metrics of fixed volume, we have 
$$\lambda_1(g_t) \leq \lambda_1(g_0)+o(t),\,\,\,\,t \rightarrow 0$$ 
\end{remark}
A smooth connected compact Riemannian manifold $(M,g)$ is called a $\lambda_1-$maximal manifold if the metric $g$ realizes the supremum in $\Lambda(M)$.

For instance, in Hersch's result (see \cite{hersch}), $\mathbb S^2$ endowed with a round metric is actually $\lambda_1-$maximal. Similarly, $\mathbb R P^2$ with its standard metric is $\lambda_1-$maximal (see \cite{LiYau}) and the flat equilateral torus is the only $\lambda_1-$maximal torus (see \cite{N1}). This latter fact induces some consequences on the Berger's isoperimetric problem (see \cite{berger,N1}).

On the other hand, an isometric immersion $\varphi$ from $(M,g)$ in the sphere is a minimal immersion if and only if it satisfies for some $\lambda$
$$-\Delta_g \varphi=\lambda \varphi. $$
 
If $\lambda$ is the first eigenvalue of the laplacian then the manifold $(M,g)$ is said to be $\lambda_1-$minimal.  For instance, any Riemannian irreducible homogeneous space is $\lambda_1-$minimal. In \cite{N1}, the first author proved the following result: any $\lambda_1-$maximal Riemannian surface is $\lambda_1-$minimal. This result has been generalized by El Soufi and Ilias to any dimension in \cite{SIPacific}. The importance of maximal metrics in Riemannian geometry is related to $\lambda_1-$minimality. The metric $g$ on an $n-$dimensional manifold is $\lambda_1-$minimal if the eigenspace $U_1(g)$ associated to the first non zero eigenvalue of the Laplace-Beltrami operator contains a family $\left \{ u_1,\cdot \cdot \cdot, u_k \right \}$ of eigenfunctions such that 
\begin{equation}\label{metric}
g=\sum_{i=1}^k du_i \otimes du_i.
\end{equation}

It appears that the topological spectrum has deep connections with minimal submanifolds of Euclidean spheres.
Indeed, by a well-known result of Takahashi (see \cite{taka}), the map $U: M \to \mathbb R^k$ defined by $U=(u_1,\cdot \cdot \cdot, u_k)$
is a minimal immersion from $(M,g)$ into the Euclidean sphere $\mathbb  S^{k-1}$ if and only if the metric $g$ writes as in \eqref{metric}. 

The hardest question on the existence of a smooth, or at least sufficiently smooth, metric maximizing the first eigenvalue reminded open. In a natural ramification of this problem, one can consider a topological spectrum under additional constraints of staying in the conformal class of the background metric. This leads to the so-called conformal spectrum. We define  
$$\tilde \Lambda(M, [g])=\sup_{\tilde g \in [g],\,\,A_{\tilde g}(M)=1} \lambda_1(\tilde g) $$
where $[g]$ is the conformal class of $g$.  Recently, a lot of attention has been devoted to the conformal spectrum on surfaces. For instance, isoperimetric inequalities have been obtained in \cite{colboisSoufi,IliasRosSoufi} in a conformal class context.  Li and Yau (see \cite{LiYau}) also discovered a bound between the conformal spectrum (the first eigenvalue) and the conformal volume. The following important inequality was proved in \cite{colboisSoufi}
$$\tilde \Lambda(M,[g]) \geq 8 \pi. $$

The central purpose of the present paper is two fold: first we want to construct in a general class of surfaces $M$ an extremalizing metric in the conformal class; second, we want to establish a link between the conformal spectrum and the harmonic maps of the surface into the Euclidean spheres. We prove the existence of an extremalizing metric for $\tilde \Lambda(M,[g])$. 

Our construction is rather explicit in the sense that it is based on an approximation procedure. We prove that there exists a smooth and positive, up to a finite discrete set of points on $M$ , metric in the conformal class $ [g]$ such that it maximizes $\lambda_1$ with fixed volume. This provides in a two-dimensional framework a quite complete picture  by considering the map generated by several eigenfunctions of the extremalizing metric.

We would like also to mention a recent preprint by Kokarev \cite{kokarev} devoted to similar problems. The results of Kokarev are somehow complementary to ours, though there is no direct overlapping.

\section{Notations and results}

Let $(M,g)$ be a two-dimensional Riemannian manifold. In local coordinates $(x_i,y_i)$, the metric writes $g=\sum g_{ij} dx_i dy_j$ and the Laplace-Beltrami operator has the form 
$$\Delta_g =\frac{1}{\sqrt{|g|}} \frac{\partial }{\partial x_i} \Big ( \sqrt{|g|}g^{ij}\frac{\partial }{\partial y^j} \Big )$$
where we have used the usual convention of repeated indexes and $g^{ij}=(g_{ij})^{-1}$, $|g|=det(g_{ij})$. 

We denote by $\lambda_1(g)$ the first non-zero eigenvalue of $\Delta_g$ and we have 
$$\lambda_1(g)=\inf_{ u \in E} R_{M,g}(u)$$

where $R_{M,g}(u)$ is the so-called Rayleigh quotient given by 
$$R_{M,g}(u)=\frac{\int_M |\nabla u|^2 dA_g}{\int_M u^2 dA_g}$$

and the infimum is taken over the space 
$$E= \left \{ u \in H^1(M),\,\,\,\,\int_M u=0 \right \}. $$

Due to the scaling property of the first eigenvalue under a metric change $cg$, it is natural to introduce a normalization for the metric. We then consider on $M$ the class $[g]$ of metrics conformal to $g$, i.e. 

$$[g]=\left \{ g',\,\,A_{g'}(M)=1,\,\,g'=\mu\,g,\,\,\mu >0,\,\,\mu:M \to \mathbb R, \mu \in L^1(M) \right \}.$$

\begin{remark}
Notice that in the previous definition, we do not make any {\sl a priori }assumption on the regularity of the map $\mu$, except of its summability, which is a requirement to fix the volume.  
\end{remark}

In dimension $2$, the Laplace-Beltrami operator is conformally covariant in the following sense: if $g'=\mu g$, we have 
$$(-\Delta_{g'})=\frac{1}{\mu} (-\Delta_g)$$

and the surface element is conformally changed by the law 
$$dA_{g'} =\mu dA_g.$$

As before, define 
$$\tilde \Lambda(M,[g])= \sup_{ g' \in [g]} \lambda_1(g'). $$

We state now our results. We first prove existence and regularity of a maximizing metric. 

\begin{theorem}\label{main1}
Let $(M,g)$ be a smooth connected compact boundaryless Riemannian surface. Assume that
$\tilde \Lambda(M,[g])> 8\pi . $ Then there exists a metric $\overline g \in [g],$  $\overline g=\mu  g$, where $\mu $ is a smooth function positive outside a finite number of points,  such that the metric $\overline g$ extremalizes the first eigenvalue in the conformal class of $g$, i.e. 
\begin{equation}\label{extrem}
\lambda_1 (\overline g)=\tilde \Lambda(M,[g]). 
\end{equation}
\end{theorem}

Theorem \ref{main1} implies the following characterization of the metric. 
\begin{theorem}\label{main2}
Let $(M,g)$ be a smooth connected compact boundaryless Riemannian surface. Assume that
$\tilde \Lambda(M,[g])> 8\pi$ and denote $\overline g$, the maximizing metric obtained in Theorem \ref{main1}. Denote $U_1(\overline g)$ the eigenspace associated to $\lambda_1 (\overline g)$. Then there exists a family of eigenvectors $\left \{ u_1,\cdot \cdot \cdot,u_\ell \right \} \subset U_1(\overline g)$ such that the map  
\begin{equation}
\left \{
\begin{array}{c}
\phi: M\to \mathbb R^\ell\\
x \to (u_1,\cdot \cdot \cdot,u_\ell)
\end{array} \right. 
\end{equation}  
is an harmonic map into the sphere $\mathbb S^{\ell-1}$.

\end{theorem}

Theorem \ref{main2} is a generalization of the results in the paper \cite{soufi} to the metrics with singular points. Instead of trying directly to extend the methods of \cite{soufi} to the singular context, we provide an alternative argument. Theorem \ref{main2} admits the following corollary. 
\begin{cor}\label{corMain}
Let $(M,g)$ be a smooth connected compact boundaryless Riemannian surface. Assume that
$\tilde \Lambda(M,[g])=\Lambda(M)> 8\pi . $ Then there exists a metric $\overline g $ smooth outside a finite number of conical singularities such that the map  
\begin{equation}
\left \{
\begin{array}{c}
\phi: M\to \mathbb R^\ell\\
x \to (u_1,\cdot \cdot \cdot,u_\ell)
\end{array} \right. 
\end{equation}  
is a harmonic map into the sphere $\mathbb S^{\ell-1}$. If $\ell >2$ , the map $\phi$ is a minimal branched conformal immersion into $\mathbb S^{\ell -1}$. 
\end{cor} 
In the previous statement, by "finite number of conical singularities" , we mean: there exist $K \in \mathbb N^*$ and $\left \{p_k \right \}_{k=1,...,K} \in M^K$ such that the density $\mu$ of $\overline g$ satisfies:
\begin{itemize}
\item $\mu \in C^\infty( M ) $. 
\item $\mu >0$ on $M \backslash \left \{p_1,...,p_K \right \})$.
\item $\mu $ has a finite order of vanishing at the points $p_k $. 
\end{itemize} 

\begin{proof} By Theorem \ref{main1} there exists a metric $\overline g \in [g]$ , such that
$\lambda_1 (\overline g)=\tilde \Lambda(M,[g]) $. Since we assume
$$ \tilde \Lambda(M,[g])=\Lambda(M),$$
 the metric  $\overline g$ extremalizes $\lambda_1 $ also with respect to variations of the conformal class of the metric. Hence by the result of \cite{N1} the
corollary follows (see also \cite{colboisSoufi}).

\end{proof}

The proofs of the previous theorems rely on a careful analysis of a Schr\"odinger type operator. Indeed consider $g' \in [g]$, by conformal invariance, the equation $-\Delta_{g'} u= \lambda_1(g') u$ reduces to the following system 
\begin{equation}\label{problem}
\left \{ 
\begin{array}{c}
-\Delta_g u =\lambda_1(g')\, \mu \,u ,\,\,\,\mbox{on}\,\,M\\
\int_M \mu \, dA_g=1. 
\end{array} \right . 
\end{equation}

We cannot assume from the beginning that the extremalizing metric $\mu\, g$ belongs to the smooth category but instead we will prove that this is the case up to a finite number of conical singularities. The strategy of the proof is the following: 
\begin{enumerate}
\item We first regularize the problem by considering an extremalizing sequence of densities in a space of probability measures with bounded densities of indefinite sign.   We then carefully estimate the "bad" sets where the densities might have some inappropriate behaviour.
\item We then prove {\sl a priori} regularity results.  
\item We then pass to the limit. 
\end{enumerate}

\section{Construction of an extremalizing sequence of metrics and estimates }

This section is devoted to the construction of an extremalizing sequence of metrics for problem \eqref{problem}. To do so, we perform a regularization by considering it as limit of bounded densities. More precisely, denote by $S_N$ the class of densities $\mu $ such that $-\frac12\leq \mu \leq N$, $\int_M \mu dA_g=1$ for $N>0$. Denoting by $\lambda_1(\mu)$ the eigenvalue problem in \eqref{problem} with a density $\mu \in S_N$, we write 
$$\tilde \Lambda_N =\sup_{\mu \in S_N} \lambda_1(\mu). $$

We first construct a sequence of densities converging to $\tilde \Lambda_N$.  
\begin{pro}
For any given $N>0$, there exists a sequence $\left \{ \mu_{k,N} \right \}_{k \geq 0} \subset S_N$ such that as $k\rightarrow +\infty$

$$\mu_{k,N} \rightharpoonup^*  \mu_N \,\,\,\,\text{weakly in measure} $$

and 

$$\lambda_1(\mu_{k,N}) \rightarrow \tilde \Lambda_N .$$

Furthermore, we have  
 $$\int_M \mu_N\,dA_g=1$$
and 
$$ -\frac12 \leq \mu_N  \leq N.$$
\end{pro}
\begin{proof}
Fix $N>0$. By the well-known universal bounds for the first non zero eigenvalue for Schr\"odinger operators with bounded potential (see \cite{lieb}), one deduces the existence of the sequence $\mu_{k,N}$. Since weak convergence of bounded potentials of Schr\"odinger operators leads to strong convergence of the solutions, it immediately implies the following result. 
\end{proof}
The whole point now is to pass to the limit $N \rightarrow + \infty$ and to prove that the limit obtained this way is indeed a nonnegative density, with sufficient regularity. This amounts to control the two following subsets of $M$ 

$$E^N_{-}= \left \{ x \in M,\,\,\,-\frac12 \leq \mu_N(x) \leq 0 \right \}$$
and
$$E_N= \left \{ x \in M,\,\,\,\mu_N(x)=N \right \}.$$
 
We denote for once and for all $A_g(\Omega)$ the area with respect to the metric $g$ of a domain $\Omega \subset M$. We also denote $B(0,r)$ the ball in $\mathbb R^n$ of center $0$ and radius $r$ and $S(0,r)=\partial B(0,r)$.

\subsection{Measure estimates}

We have first the following easy lemma.  

\begin{lemma}\label{measEN}
Let $N>0$. Then there exists $C>0$ such that  
$$A_g(E_N) \leq C/N. $$
\end{lemma}
\begin{proof}
The density $\mu_N \,g$ is of the integral one, i.e. 
$$\int_M \mu_N \, dA_g =1.$$
Writing
$$\int_M \mu_N \, dA_g= \int_{E_N}\mu_N \, dA_g + \int_{M \backslash E_N}\mu_N \, dA_g, $$

leads
$$\int_M \mu_N \, dA_g= N A_g(E_N) + \int_{M \backslash E_N}\mu_N \, dA_g. $$

We then have since $\mu_N \geq -1/2$
$$\int_{M \backslash E_N}\mu_N \, dA_g > -\frac{1}{2} A_g (M \backslash E_N). $$

Writing $A_g (M \backslash E_N)= A_g(M)-A_g (E_N)$ gives the desired result. 

\end{proof}

The following lemma is also well known (see \cite{lieb}, Th. 12.4). 
\begin{lemma} \label{schroGround}
 Let $v$ be a solution of the Dirichlet problem 
 \begin{equation}\label{temp}
\left \{ 
\begin{array}{c}
-\Delta v=Vv,\,\,\,\mbox{in}\,\,B(0,1)\\
v=0 ,\,\,\,\mbox{on}\,\,S(0,1)
\end{array} \right . 
\end{equation}
Let  $p>1$. Then there exists $C=C(p)>0$ such that $\|V^+\|_p\leq C$ implies $v\equiv 0$,. Here $V^+$ denotes the positive part of the potential $V$.  
\end{lemma}
\begin{proof} 
If $C( p)>0$ is a sufficiently small constant then the first eigenvalue of the Schr\" odinger
operator with the potential $V$ is positive. Correspondingly the Dirichlet problem \eqref{temp} has a unique solution, hence the result. 
\end{proof}

We now come to the measure estimate of the set $E^N_{-}$.  We have the following general lemma. This is actually proved in any dimension in our paper \cite{GNS} to which we refer (Lemma 3.3). 
\begin{lemma} \label{measurePlane}
For any $N>0$, there exists $\epsilon_0 =\epsilon_0 (N)>0$ such that: for every $\epsilon <\epsilon_0$, if $E\subset B(0,1)\subset \RR^2$ is a measurable set satisfying 
$$ |E|< \epsilon $$ 
and $v>0$ is a positive weak solution in $B(0,1)$ of the differential inequality
\begin{equation}\label{schro}
-\Delta v - hv \leq 0,
\end{equation}
 $h \in L^\infty(B(0,1))$ satisfying  
\begin{equation}\label{h}
\left \{
\begin{array}{cc}
|h|<N \,\,\mbox{on}\,B(0,1),\\
h<-1/N ,\,\mbox{on}\,B(0,1)\setminus E,
\end{array} \right . 
\end{equation}
then the following holds 
\begin{equation}\label{mean}  
v(0)< \frac{1}{2\pi}\int_{S(0,1)}v d\sigma.
\end{equation}
\end{lemma}

As an immediate corollary of Lemma \ref{measurePlane} we have
\begin{lemma} \label{corLem}
For any  $N>0$ there exists  $\epsilon_0 =\epsilon_0 (N)>0$ such that for every $\epsilon <\epsilon_0$, if $E\subset B(0,2)\subset \mathbb R^2$ is a measurable set such that
$$| E| < \epsilon $$ 
and $v>0$  is a solution in $B(0,2)$ of the following
differential inequality
\begin{equation}
-\Delta v - hv \leq 0,
\end{equation}
where  $h$ satisfies \eqref{h}
then 
$$v(0)< {1\over 3\pi }\int_{B(0,2)\setminus B(0,1)}vdx  $$

\end{lemma}

Considering a local conformal structure on $M$ we can lift the last lemma on $M$. This is the purpose of the following lemma.

\begin{lemma}\label{measureMani} 
There exists $r_0$ such that for each $x\in M$ and $0<r<r_0$, then in $G_x=B(x,r)\backslash B(x,r/2) $ (where $B(x,r)$ is a geodesic disk), there exists a function $q\in C(B(x,r)),\ q>0$
such that for any $N>0$ there exists  $\epsilon =\epsilon (N)>0$ such that if $ E\subset B(x,r)$ being a measurable set with 
$$A_g(E)< \epsilon $$ 
and $v>0$  be a solution in $B(x,r)$ of the following
differential inequality
\begin{equation}\label{schroMani}
-\Delta v - hv \leq 0,
\end{equation}
where $h \in L^\infty(B(x,r))$ satisfies \eqref{h}
then 
$$v(x)< \frac{\int_{G_x}qvdA_g}{\int_{G_x}qdA_g}.  $$

\end{lemma}

\begin{proof} Let $\psi $ be a conformal map from $B(x,r)$ on $B(0,1) \subset \mathbb R^2$. Taking $q$ to be the Jacobian of $\psi $, the thesis follows from Lemma \ref{corLem}. 
\end{proof}

We prove now the following lemma. 

\begin{lemma} \label{negSet}
For any $N>0$, we have 
 $$A_g(E^N_-)=0.$$
\end{lemma}

\begin{proof} 
Fix $N>0$. The result is an easy consequence of the following fact: let $E_\delta$ be the set 
$$E_\delta=\left \{ x \in M,\,\,| \,\,-\frac12 \leq \mu_N(x)\leq -\delta \right \},$$
where $\delta \in (0,1/2)$. Then for any $\delta \in (0,1/2)$ 
$$A_g(E_\delta)=0.$$ 
We argue by contradiction and assume that for some $\delta$
$$A_g(E_\delta) >0. $$
Denote $E= M\backslash E_\delta$ and $\Sigma$ the set of Lebesgue points of $E$. Recall that almost all points of a measurable set are Lebesgue points. For each $x\in \Sigma $
denote by $B_x$ the disk centered at $x$ such that
$$\frac{A_g(E \cap B_x)}{A_g(B_x)}<\varepsilon.$$
Let $G_x\subset B_x $ be the set defined in Lemma \ref{measureMani} and define on $ B(x,r) $ the quantities 
$$g'=qg$$
and 
$$f_x(y)=q \frac{\chi(G_x)(y)}{A_{g'}(G_x)} $$
where $\chi(A)$ is the characteristic function of the set $A$.  We introduce the following integral operator $T$: 

\begin{equation}
\begin{array}{c}
T: L^1(M) \mapsto L^1(M) \\
\\
T(h)=\int_\Sigma h(x) f_xdA_{g}+\tilde h 
\end{array}
\end{equation}
where

\begin{equation*}
\tilde h =\left \{ 
\begin{array}{c}
0\,\,\,\,\mbox{on}\,\,\Sigma,\\
h\,\,\,\,\mbox{on}\,\,M \backslash \Sigma . 
\end{array} \right . 
\end{equation*}  

The operator $T$ preserves the $L^1$ norm of positive functions on $M$, i.e. for all $h \in L^1(M)$, $h\geq 0$
$$\int_M T(h)(y) dA_{g}=\int_M h(y) dA_{g} . $$

As a consequence for any $n \geq 1$, we have 
$$\int_M T^n(\varphi )dA_{g}=1$$
for any  $\varphi \in L^1(M)$, such that $\varphi \geq 0$, $\int_M \varphi dA_g=1$ .
Set
$$h=\chi ( \Sigma ) /A_{g}(\Sigma ).$$

Then the sequence $\{T^n(h) dA_{g}\}_n $ is a sequence of probability measures which contains a subsequence of measures weakly converging to a measure $h^*$ and the measure $h^*$ is supported on $M\setminus \Sigma $. Let $u$ be a solution of \eqref{problem}. Then the function $v=u^2$ satisfies
$$
-\Delta v-\lambda_1 \mu v \leq 0.
$$
Therefore, applying this to the potential $\mu=\frac{\mu_N}{2\lambda_1}$, the function $v=(u^N)^2$ satisfies \eqref{schroMani} in $B(x,r)$ with $E=E_\delta$ and $\delta=1/N$ for $N>2$. Hence for $x\in \Sigma $, by Lemma \ref{measureMani}, one has (setting $u^N=u$)
$$u^2(x)<\int_M u^2(y)f_x(y)dA_{g} .$$
Thus by definition of the operator $T$, we have 
$$\int_M u^2(y)h(y)dA_{g}<\int_M u^2(y)h^*(y)dA_{g} .$$
Since the functions $u^N$ are uniformly continuous, it follows that we can approximate
the measure $h^*$ by a function $s\in L^{\infty }$ such that $s\geq 0$ , has support on $M\setminus \Sigma $, 
$$\int_M sdA_{g}=1$$
and
$$\int_M u^2(y)h(y)dA_{g}<\int_M u^2(y)s(y)dA_{g} .$$
Denote
$$K= ess\sup_M s +ess\sup_M h,$$ 
$$p(x)= (h(x) -s(x))/2Kq(x) .$$
Then we have 
$$\int_M u^2pdA_g < 0.$$

Setting $\overline \mu_N = \mu_N+(A_g(G_x))p$, we have for $N$ large enough
$$-\frac{1}{2} < \overline \mu_N < N $$
and the mesure $\bar \mu_N$ is admissible. On the other hand, we have  
$$\int_M v^2 \overline \mu_N dA_g < \int_M v^2 \mu_N dA_g$$
which increases $\tilde \Lambda_N$, a contradiction with the optimality of $\mu_N$. The lemma is proved. 
\end{proof}

\subsection{Control of the eigenfunctions}

We start with the following general lemma.  
\begin{lemma}
Let $E \subset M$ be a domain in $M$. Let $Q$ be a convex  cone in $L^2(M)$ such that if $v \in  Q$ then $v \geq 0$. Assume that for all $\varphi \in L^2(M)$ such that $\int_M \varphi =0$ and $\varphi \geq 0$ on $E$, there exists $q \in Q $ such that $\int_M \varphi q \geq 0$. 

Then there exists $\tilde q \in Q$ such that 
\begin{enumerate}
\item $\tilde q \equiv 1$ on $M \backslash E$
\item $\int_M \tilde q \leq 1$. 
\end{enumerate}
\end{lemma}
\begin{proof}
Denote by $\mathcal E$ the convex set
$$\mathcal E= 1^\perp \bigcap Q$$
where $1^\perp$ denotes the hyperplane $\left \{ u \in L^2(M) \,|\,\int_M u =0 \right \}. $ Denote by $K$ the convex cone 
$$K =\left \{ u \in Q \, |\, u \equiv 0\,\,\,\mbox{on}\, E\,\, , \,\, \int_M u \leq 0 \right \}. $$
The claim of the theorem amounts to prove that 
$$(K+1) \bigcap \mathcal E \neq \emptyset. $$

Assume that this is not the case, i.e. $(K+1) \bigcap \mathcal E = \emptyset$. Since $\mathcal E$ and $K+1$ are two closed convex sets in $L^2(M)$, by Hahn-Banach theorem, there exists a hyperplane $\mathcal H$ separating $\mathcal E$ from $K+1$. Let $ \bar n$ be a normal vector to the hyperplane $\mathcal H$. We claim that $ \bar n$ satisfies the three following properties 
\begin{itemize}
\item $\int_M \bar n =0$. 
\item $\bar n \geq 0$ on $E$. 
\item For all $q \in Q$, we have $\int_M q \bar n < 0$. 
\end{itemize}  
Therefore, it contradicts the assumptions of the theorem, hence the result. The first point of the claim comes form the construction of $\bar n$. For the second and third points of the claim, it suffices to notice that, from standard convex analysis, $n$ belongs to the polar cone of $K+1$, i.e. 
$$(K+1)^*= \left \{ u \in L^2(M)\,|\, u\geq 0\,\,\,\mbox{on}\, E\,\, , \,\, \int_M u q < 0, \forall q \in Q \right \}.$$ 
\end{proof}
In our context, the previous lemma admits the following corollary. 
\begin{cor}\label{eigenfunctions}
Denote $\bar g_N =\mu_N g.$ Let $U_1(\bar g_N)$ be the eigenspace associated to $\lambda_1$, i.e. the set of functions satisfying 
$$-\Delta_g u = \lambda_1 \, \mu_N \, u .$$ 
Then there exists an orthogonal family $ \left \{ u_1^N,\cdot \cdot \cdot, u_\ell^N  \right \} \subset U_1(\bar g_N)$ such that if we denote $w= \sum_{i=1}^\ell (u_i^N)^2$ then 
\begin{enumerate}
\item $w \equiv 1$ on $M \backslash E_N$
\item $\int_M w \mu_N dA_g \leq 1$
\end{enumerate}
where $E_N= \left \{ x \in M \,\,|\,\,\mu_N(x)=N \right \}.$
\end{cor}
\begin{proof}
First notice that $A_{\bar g_N}(M)=1$. We denote 

$$\hat Q= \left \{ u^2; u \in U_1(\bar g_N)\right \}. $$

Let $Q$ be the convex enveloppe of the cone $\hat Q$. To be able to apply the previous lemma, we just need to check that for all $\varphi \in L^2(M)$ such that $\int_M \varphi =0$ and $\varphi \geq 0$ on $E_N$, there exists $q \in \hat Q $ such that $\int_M \varphi q \geq 0$. Assume the contrary, i.e. there exists $\tilde \varphi \in L^2(M)$ such that $\int_M \tilde \varphi =0$, $\tilde \varphi \leq 0$ on $E_N$ and for all $q \in \hat Q $, $\int_M \tilde \varphi q < 0$. Consider the new potential
$$\tilde \mu_N= \mu_N +\varepsilon \tilde \varphi. $$

Therefore, on $E_N$, since $\tilde \varphi \leq 0$ we have 
$$\tilde \mu_N=N+\varepsilon \tilde \varphi \leq N$$
and $\tilde \mu_N$ is an admissible potential for $\varepsilon$ small enough. Dropping the notations $i,N$ from the eigenvectors, we estimate 

$$R_{M,\tilde \mu_N g} (u)- R_{M,\mu_N g}(u).$$

Since $\int_M |\nabla u|^2dA_g$ is invariant in dimension 2 under a conformal change, one has 

$$R_{M,\tilde \mu_N g} (u)- R_{M,\mu_N g}(u)=\int_M |\nabla u|^2dA_g\times$$
$$\Big ( \frac{1}{\int_M \tilde \mu_N u^2dA_g}-\frac{1}{\int_M\mu_N u^2dA_g}\Big ).$$

Hence 
$$R_{M,\tilde \mu_N g} (u)- R_{M,\mu_N g}(u)= \alpha(u)\Big ( \int_M  \mu_N u^2dA_g-\int_M \tilde \mu_N u^2dA_g\Big )$$
where $\alpha(u)>0$. This gives 
$$R_{M,\tilde \mu_N g} (u)- R_{M,\mu_N g}(u)=\alpha(u)\Big (-\varepsilon \int_M \tilde \varphi u^2 \Big ). $$

By assumption on $\tilde \varphi$, we have that 
$$-\varepsilon  \int_M \tilde \varphi u^2 > 0$$
hence we have 

$$R_{M,\tilde \mu_N g} (u)> R_{M,\mu_N g}(u).$$
On the other hand, the potential $\mu_N$ is the one realizing the extremum of the first eiigenvalue, hence a contradiction. The lemma is proved.

\end{proof}

\section{Proofs of Theorems \ref{main1} and \ref{main2}}

We now reach the conclusions of our Theorems \ref{main1} and \ref{main2}. Since the measures $\mu_N$ are uniformly bounded by Lemma \ref{measEN}, we have the following convergences: 
$$\mu_N \rightharpoonup^* \mu ,\,\,\,\text{weakly in measures},$$
$$u_i^N \rightharpoonup u_i,\,\,\,\text{weakly in}\,H^1(M),i=1,...,\ell.$$

Furthermore, by the continuity in weak topology of the first eigenvalue with respect to the metric, we have 
$$ \tilde \Lambda_N \rightarrow \tilde \Lambda= \tilde \Lambda(M,[g]) .$$

Finally, by Lemma \ref{negSet}, the weak limit $\mu$ satisfies distributionally 
$$\mu >0 \,\,\mbox{a.e. in }M .$$

 First we exclude that the limiting density blows up at a point.  The following result was proved
 by A. Girouard in \cite{girouard} but for the convenience of the reader we give a proof in the Appendix. 
\begin{thm}\label{dirac}
Assume that $\tilde \Lambda(M,[g])> 8\pi . $ Then the measure $\mu\, dA_g$ is not a Dirac measure. 
\end{thm}

This implies the following lemma, proved in \cite{N1}, section 4.4. 
\begin{lemma}\label{mespoint}
Let $x_0\in M$. Then $\mu (x_0)=0$.
\end{lemma}

We prove the following fact. 
\begin{lemma}\label{templem}
Let $G\subset \mathbb R^2$ be a bounded domain. Let $w\in C^2(G)$ satisfy
$$\Delta w =k_1(x)w +k_2(x)\quad in \quad G,$$
$$w=b(x)\quad on \quad \partial G,$$
where $|k_1|,|k_2|<K,\, |b|<B$ for some positive constants  $K,B$. 

Then there exists $\delta=\delta(K,B)>0$
such that if $|G|<\delta$ the following bound holds
$$|w|<2B\quad in \quad G.$$
\end{lemma}

\begin{proof}
Let $D\subset \mathbb R^2$ be a disk equimesurable with $G$ such that $|D|=\delta$. We consider the Dirichlet problem in $D$,
\begin{equation}\label{D.P}
\left \{
\begin{array}{c}
\Delta v =-Kv-K \quad in \quad D\\
v=B \quad on  \quad \partial D
\end{array} \right. 
\end{equation}  
For sufficiently small $\delta >0$ the problem \eqref{D.P} has a unique solution $v$. 

Let $\tilde k_i,\tilde b$ be the spherical rearrangements of $k_i$ and $b$. Since $\tilde k_i<K,\,
\tilde b<B$ then by comparison results for spherical rearrangements of the Dirichlet
problem \cite{TV}, we have
$$|w|<v(0).$$
Hence for sufficiently small $\delta >0$ we get  the desired result. 

\end{proof}

By Lemma \ref{measEN}, we have
$$A_g(E_N)<C/N$$
where the constant $C>0$ depends on the genus of the surface $M$. Therefore from Lemma \ref{mespoint},
it follows that for any $x_0\in M, \, \epsilon >0$ there exist $r>0, \, N_0>0$ such that for $N>N_0$
\begin{equation}\label{meas}  
A_g(E_N\cap B(x_0,r))<\epsilon/N.  
\end{equation}  

Denote 
$$\phi_N:(u^N_1,\dots ,u^N_\ell)\to \mathbb R^\ell$$

We may assume without loss of generality that  $w^N= \sum_{i=1}^\ell (u_i^N)^2\to 1$ a.e. on $M$. Furthermore, by Corollary \ref{eigenfunctions} it follows that
$$\|w^N\|_{L^2(M)}\leq 1.$$

Let $x_0\in M$ and denote $F_a\subset (0,1),\, a>0$ the set such that if $r\in F_a$ then 
$$\liminf_{N\to \infty} \mbox{diam} \, \phi_N(\partial B(x_0,r))>a.$$

We have 

\begin{lemma}\label{ptemeas}
For any $a>0, r>0$ the set $(0,r)\setminus F_a$ is non-empty.
\end{lemma}
\begin{proof}
Assume that $(0,r)\subset F_a$. This yields that 
$$\sum_{i=1}^\ell \|\nabla u_i^N\|_{L^2}^2 \to \infty\, ,\,N\to \infty .$$
Since the Dirichlet integrals of $u_i^N$ are uniformly bounded for all $N$,
the lemma is proved.
\end{proof}

\begin{lemma}\label{capMeas}
We have 
$$\lim_{N \to +\infty} w^N=1\,\,\,\mbox{on}\,\,M.$$
\end{lemma}
\begin{proof}
Let $x_0\in M$. By Lemma \ref{ptemeas} for any $a>0, \epsilon >0$ there is $0<r<a$ such that
$$\liminf_{N\to \infty} \mbox{diam} \, \phi_N(\partial B(x_0,r))< \epsilon .$$
Choosing if needed  a subsequence, we may assume without loss that
$$\limsup_{N\to \infty} \mbox{diam} \, \phi_N(\partial B(x_0,r))< \epsilon .$$
The last inequality implies that for all sufficiently large $N$ there exists a function $v_N$
in the span of $\{u_1^N,\dots , u_\ell^N\}$ such that 
$$|1-v_N|<C\sqrt \epsilon \quad on \quad \partial B(x_0,r),$$
where $C>0$ is a constant. Moreover if $u_i^N\perp v_N$ then
$$|u_i^N|<C\sqrt \epsilon \quad on \quad\partial B(x_0,r).$$
 Applying Lemma \ref{templem} to the functions $w=u^N_i$ and 
$w=1-v_N$ for sufficiently large $N$ we get from inequality \eqref{meas}
 \begin{equation} \label{v}
|1-v_N|<C\sqrt \epsilon \quad in \quad B(x_0,r),
\end{equation}
and
 \begin{equation} \label{u}
|u_i^N|<C\sqrt \epsilon \quad in \quad B(x_0,r),
\end{equation}
and since $\epsilon >0$ can be chosen arbitrary small, the lemma follows.

\end{proof}

As a corollary, we have

\begin{lemma}\label{h1capaver}
There exists a subsequence $\left \{ u_i^{n_k} \right \}_{k \geq 0}$ such that 
$$u_i^{n_k} \rightarrow u_i $$
uniformly on $M$ and weakly in $H^1(M)$ as $k \rightarrow +\infty$. Moreover
$$\sum_{i=1}^\ell u_i^2=1.$$
\end{lemma}

The next lemma provides the eigenfunctions associated to the limiting measure $\mu$. 

\begin{lemma}\label{weakEigen}
The functions $u_i$ are eigenfunctions in a weak sense: for all $\varphi \in H^1(M)\cap L^{\infty }(M)$ and $i=1,...,\ell$ we have 
 
$$\int_M \nabla u_i \cdot \nabla \varphi\,dA_g =\tilde \Lambda \int_M \mu u_i\, \varphi\, dA_g. $$\end{lemma}
\begin{proof}
By definition, we have for all $\varphi \in H^1(M)$, $i=1,...,\ell$ and any $N>0$
\begin{equation}\label{weak}
\int_M \nabla u^N_i \cdot \nabla \varphi\,dA_g =\tilde \Lambda_N \int_M \mu_N u^N_i\, \varphi\, dA_g.  
\end{equation}
By the weak $H^1$ convergence of the sequence $ \left \{ u^N_i \right \}_N$, we have 
$$\int_M \nabla u^N_i \cdot \nabla \varphi\,dA_g \rightarrow  \int_M \nabla u_i \cdot \nabla \varphi\,dA_g. $$ 

Hence from Theorem \ref{h1capaver} it follows that we can pass to the limit in  identity \eqref{weak} 
and get the desired result.

\end{proof}

As a corollary of Lemma \ref{weakEigen} we have

\begin{cor}\label{eigendistrib}
The following equality holds in the sense of distributions:
$$-\Delta u_i = \tilde \Lambda  \mu u_i\,\,\,\, i=1,...,\ell .$$

\end{cor}

A standard density procedure using an approximation of the identity gives the following result. 

\begin{lemma}\label{harmonic}
The following equality holds in the sense of distributions:
$$\Delta u_i^2= 2|\nabla u_i|^2 -  \tilde \Lambda \mu u_i^2\,\,\,\, i=1,...,\ell .$$

\end{lemma}

We now come to the proofs of the main theorems. 

\subsection*{Proof of Theorem \ref{main1}}
 From Morrey's regularity result \cite{morrey}, Lemma \ref{weakEigen} and Corollary \ref{eigendistrib}, it follows that all the eigenfunctions $u_i$ are real analytic. Hence the density $\mu $ is a real analytic
function on $M$, positive outside an analytic manifold $\Gamma $. Since $\mu $ is not identically 
zero it follows that the dimension of $\Gamma $ is either $0$ or $1$. Assume that the dimension of 
$\Gamma $ is $1$. Then from formula \eqref{mu} it follows that $|\nabla u_i|=0$ on $\Gamma $.
Let $\Gamma_0 $ be a connected component of $\Gamma $. 
Then it follows that all $u_i$ are constants on $\Gamma_0$. Assume without loss of generality that $u_1= 0$
on $\Gamma_0 $. Thus $u_1=|\nabla u_1|=0$ on $\Gamma_0$ and by uniqueness, it follows that $u_1$ is identically zero. Hence the dimension of $\Gamma $ is $0$ and thus $\Gamma $ is a set of at most  a finite number of points on $M$, hence the theorem. 
\vspace{1cm}
\subsection*{Proof of Theorem \ref{main2}}

 From Lemma \ref{h1capaver}, one has that
 $$\sum_{i=1}^\ell u_i^2=1\,\,\,\mbox{on}\,\,M. $$
Therefore, the map $\phi=(u_1,\cdot \cdot \cdot, u_\ell): M \to \mathbb S^{\ell -1}$ is well defined a.e. on $M$. Applying the Laplace-Beltrami $\Delta $ to the last identity and by Lemma \ref{harmonic}, it follows that
 $$\sum_{i=1}^\ell \tilde \Lambda  \mu u_i^2 - \sum_{i=1}^\ell |\nabla u_i |^2 = 0.$$
 Thus
 \begin{equation} \label{mu}
 \mu = \sum_{i=1}^\ell |\nabla u_i |^2 / \tilde \Lambda .
 \end{equation}
 
 The last equality implies that the limit measure $\mu $ has an $L^1$ density and 
 moreover the map $\phi $ is a weak solution of the harmonic map equation, namely
 $$\sum_{i=1}^\ell u_i \Delta u_i = \sum_{i=1}^\ell |\nabla u_i|^2 . $$ 
 
 Hence by the result of H\'elein \cite{helein}, the map  $\phi $ is a smooth harmonic map of $M$ into the sphere. In the appendix, we give an alternative proof of the fact that $\phi$ is an harmonic map.

\section*{Acknowledgements}
This work was initiated during a long stay visit at the Laboratoire Poncelet, Moscow which hospitality is warmly acknowledged. 

\section*{Appendix}
We prove here that the map $\phi$ defined by $\phi=(u_1,...,u_\ell)$ from $M$ into $\mathbb S^{\ell-1}$ is harmonic. We prove
 \begin{lemma}
The map $\phi=(u_1,...,u_\ell)$  is  a minimizing harmonic from $M$ into $\mathbb S^{\ell-1}$, i.e. $\phi$ minimizes in $H^1(M, \mathbb S^{\ell-1})$ the Dirichlet form 
$$\mathcal D (\psi)=\int_M |D \psi |^2 \,dA_{\bar g}$$  
where $\bar g$ is the critical metric. 
 \end{lemma}
\begin{proof}Suppose that the map $\phi$ is not harmonic. Therefore, for all $\varepsilon >0$, there exist  $E \subset M$ such that diam$(E) < \varepsilon$ and a map $\psi: M \mapsto \mathbb S^{\ell-1}$ such that  
$$\int_E |D\psi |^2 \mu\,dA_{ g}< \int_E |D\phi |^2\mu \,dA_{ g}$$ 
and 
$$\psi=\phi\,\,\,\,\,\mbox{on}\,\,M\backslash E. $$

We choose coordinates on the sphere $\mathbb S^{\ell -1}$ such that $\psi(E)$ is in the positive octant. In these coordinates, we still have on $E$
$$\sum_{i=1}^\ell \psi^2_i \equiv 1 $$ 
and then 
$$\int_E \sum_{i=1}^\ell \psi^2_i\mu \,dA_{g}= \int_E \sum_{i=1}^\ell u^2_i\mu \,dA_{ g}.$$

Then there is a component $\psi_i$ such that 
\begin{equation}\label{quot}
R_{M,\bar g}(\psi_i) < R_{M,\bar g}(u_i).  
\end{equation}
Set 
$$v= {\int_Mu^+_i\mu dA_g \over \int_M\psi^+_i\mu dA_g}\psi^+_i -(-u_k)^i.$$
with 
$$u_i^+ = \sup \{ 0,u_i\} .$$

Then $v\in E$ and from \eqref{quot}
$$R_{M,\bar g}(v)<R_{M,\bar g}(u_i),$$

a contradiction. 
\end{proof}

\bibliographystyle{alpha}
\bibliography{biblio}

\medskip

{\em NN} -- 
CNRS, I2M UMR 7353-- Centre de Math\'ematiques et Informatique, Marseille, France. 
 
{\tt nicolas@cmi.univ-mrs.fr}

\medskip

{\em YS} --  
Universit\'e Aix-Marseille, I2M UMR 7353, Marseille, France 

{\tt sire@cmi.univ-mrs.fr} 

\end{document}